\documentclass[a4paper,english]{article}

\usepackage[utf8]{inputenc}
\usepackage[T1]{fontenc}

\usepackage{babel}

\usepackage[final]{graphicx} 
\usepackage{numprint} 

\usepackage{geometry} 
\usepackage{xcolor} 

\usepackage{rotating}

\usepackage[nottoc, notlof, notlot]{tocbibind} 

\usepackage{verbatim} 

\usepackage{hyperref} 

\usepackage{float} 
\usepackage{placeins} 

\usepackage{caption}
\usepackage{subcaption}

\usepackage{multicol} 

\usepackage{wrapfig} 

\usepackage{mathtools}
\usepackage{amsfonts,amssymb,amsthm}

\usepackage{array}
\usepackage{multirow} 

\usepackage{cleveref} 

\theoremstyle{plain}
\newtheorem{thm}{Theorem}[section]
\newtheorem{lemme}[thm]{Lemma}
\newtheorem{prop}[thm]{Proposition}
\newtheorem{cor}[thm]{Corollary}
\newtheorem*{thm*}{Theorem}

\theoremstyle{definition}
\newtheorem{defn}[thm]{Definition}
\newtheorem*{defn*}{Definition}
\newtheorem{notat}[thm]{Notation}

\newtheorem*{conj*}{Conjecture}

\newtheorem{rem}[thm]{Remark}

\DeclarePairedDelimiter\abs{\lvert}{\rvert}

\newcommand*\conjug[1]{\overline{#1}}
\renewcommand{\Re}{\mathrm{Re}}

\renewcommand*\setminus{\ensuremath{{}-{}}}

\newcommand{\HH}{\mathbb{H}}
\newcommand{\CC}{\mathbb{C}}
\newcommand{\RR}{\mathbb{R}}
\newcommand{\ZZ}{\mathbb{Z}}
\newcommand{\QQ}{\mathbb{Q}}
\newcommand{\NN}{\mathbb{N}}

\DeclareMathOperator{\dev}{Dev}
\DeclareMathOperator{\pr}{pr}

\newcommand{\hc}[1]{\mathbb{H}^{#1}_{\mathbb{C}}}
\newcommand{\dhc}[1]{\partial_{\infty} \mathbb{H}^{#1}_{\mathbb{C}}}

\newcommand{\zz}[1]{\mathbb{Z}/{#1}\mathbb{Z}*\mathbb{Z}/{#1}\mathbb{Z}}
\newcommand{\cp}[1]{\mathbb{CP}^{#1}}

\title{On the limit set of a spherical CR uniformization}
\author{
	\textsc{Miguel ACOSTA}
	\thanks{Miguel Acosta was partially supported by the grants R-STR-8023-00-B <<MnLU-MESR CAFE-AutoFi>> and R-AGR-3172-10-C <<FNR-OPEN>>.} 
	\footnote{
		  Unité de recherche en Mathématiques,
		  Université du Luxembourg, 
		  Maison du Nombre,
		  6, Avenue de la Fonte,
		  L-4364 Esch-sur-Alzette,
		  Luxembourg.}
	}
\begin{document}
\maketitle

\begin{abstract}
	We explore the limit set of a particular spherical CR uniformization of a cusped hyperbolic manifold. We prove that the limit set is the closure of a countable union of $\mathbb{R}$-circles, is connected, and contains a Hopf link with three components; we also show that the fundamental group of its complement in $S^3$ is not finitely generated. Additionally, we prove that rank-one spherical CR cusps are quotients of horotubes. 
\end{abstract}

\section{Introduction}

 Studying the geometric structures on three-dimensional manifolds is a very powerful tool to link topological properties and geometric ones, as confirmed by the Thurston geometrization conjecture proved by Perelman in 2003. Here, we will consider geometric structures in the language of $(G,X)$-structures, as described for example in \cite{gt3m}. In this context, $X$ is a model space and $G$ is a group acting transitively and analytically on it; a geometric structure on a manifold $M$ is an atlas of $M$ with values on $X$ and transition maps given by elements of $G$.
 A case of particular interest is the one of \emph{complete} structures, that can be written as $\Gamma \backslash X$, where $\Gamma$ is a subgroup of $G$ acting properly discontinuously and without fixed points on $X$. In this case, all the information on the manifold and the structure is contained in the group $\Gamma$.
 A similar situation is the one of uniformizable structures, that arises naturally when looking at conformally flat structures or the spherical CR structures that we consider in this article. In those cases, we say that a structure is \emph{uniformizable} if it can be written as $\Gamma \backslash \Omega_\Gamma$, where $\Omega_\Gamma \subset X$ is the set of discontinuity of $\Gamma$. Observe that, in these cases, the set of discontinuity is the complement of the limit set $\Lambda_\Gamma$ of $\Gamma$.
 
 For the conformal structures arising from the boundary at infinity of the real hyperbolic space $\HH_\RR^3$, a very well-known case is the one of Fuchsian representations of a surface group. The image of these representations have round circles in $S^2$ as limit sets  and uniformize two copies of the same hyperbolic surface. By the remarkable Bers double uniformization theorem, their deformations, which are quasi-Fuchsian representations, still uniformize two copies of the surface, and the limit set remains a topological circle.
 In the light of the work of Guichard and Wienhard in \cite{guichard_anosov_2012}, this last point about the limit set is expected, since the representation of the surface group is Anosov, and this condition is open.
 
 In this article, we will focus on a particular uniformization in spherical CR geometry, which is modeled on the boundary at infinity of the complex hyperbolic plane $\hc{2}$ and has group of transformations $\mathrm{PU}(2,1)$. As in the case of conformal structures coming from the boundary at infinity of $\HH_\RR^3$, there are also analogs to the Fuchsian representations, but coming in two different flavors.
 On the one hand, a surface group can be embedded in $\mathrm{PU}(2,1)$ as a subgroup of $\mathrm{PO}(2,1)$. The corresponding representation is then called \emph{$\RR$-Fuchsian}. By a computation by Burns and Shnider in \cite[Proposition 6.1]{burns_spherical_1976}, these representations uniformize the unit tangent bundle of the corresponding surface. The corresponding limit set is the boundary at infinity of a totally real subspace of $\hc{2}$, and is called an \emph{$\RR$-circle}.
 On the other hand, it can be embedded in $\mathrm{PU}(2,1)$ as a subgroup of $\mathrm{PU}(1,1)$. The corresponding representation is then called \emph{$\CC$-Fuchsian}, and still uniformizes a circle bundle over the surface. The limit set in this case is the boundary at infinity of a complex line in $\hc{2}$, and is called a \emph{$\CC$-circle}. From this construction, Falbel and Gusevskii build in \cite{falbel-gusevskii} spherical CR uniformizations for circle bundles over surfaces of arbitrary Euler number.
 
 Apart from these Fuchsian examples and the quotients of $S^3$, we know a family of cusped hyperbolic 3-manifolds admitting spherical CR uniformizations. On the side of the manifolds, there are two different uniformizations of the Whitehead link complement: one given by Schwartz in \cite{schwartz} and a second one given by Parker and Will in \cite{parker_complex_2017a}. By deforming this last structure, we prove in \cite{acosta_spherical_2019} that an infinite family of Dehn surgeries on one of the cusps of the Whitehead link complement admit spherical CR uniformizations. The structure on the last surgery of the family is precisely the spherical CR uniformization of the Figure Eight knot complement constructed by Deraux an Falbel in \cite{falbel}. As for the groups that uniformize these cusped manifolds, they are index two subgroups of triangle groups: the discreteness of most of these representations can be proved from this viewpoint, as done by Parker, Wang and Xie in \cite{parker_wang_xie}. Recall that if $p$, $q$ and $r$ are integers $\geq 2$, the $(p,q,r)$-triangle group is the abstract group with presentation
 \[
 \left\langle 
 \sigma_1 , \sigma_2 , \sigma_3 \mid
 \sigma_1^2 = \sigma_2^2 = \sigma_3^2 = 
 (\sigma_1 \sigma_2)^p=
 (\sigma_2 \sigma_3)^q=
 (\sigma_3 \sigma_1)^r
 = \mathrm{Id}
 \right\rangle .
 \]
 
  In a more general frame, the discreteness of triangle groups into $\mathrm{PU(2,1)}$ has been widely studied, see for example the survey of Schwartz \cite{schwartz_complex_2002}, or the article of Deraux \cite{deraux_deforming_2006}, where he proves that there is an explicit representation of the $(4,4,4)$-triangle group into $\mathrm{PU}(2,1)$ whose image is a cocompact lattice.
 
 For the spherical CR uniformization of cusped hyperbolic manifolds, no much is known about the corresponding limit sets. The aim of this article is to explore the limit set of a particular uniformization of a cusped hyperbolic manifold, that can be obtained as a Dehn surgery of the Whitehead link complement. We will denote the corresponding group by $\Gamma_6$ or $\Gamma$, if there is no ambiguity.
 It is the image of a particular representation of the $(3,3,6)$-triangle group, where a parabolic element appears.
 Gathering the results of \Cref{sect:r_circle_limit_set}, if $\Gamma_6$ is the uniformization group, $\Lambda_{\Gamma_6} \subset \dhc{2} \simeq S^3$ is its limit set, and $\Omega_{\Gamma_6} = \dhc{2} \setminus \Lambda_{\Gamma_6}$, we obtain:
 
 \begin{thm}
 	We have that
 	\begin{itemize}
 		\item The limit set $\Lambda_{\Gamma_6}$ is connected and the closure in $\dhc{2}$ of a countable union of $\RR$-circles.
 		\item The limit set $\Lambda_{\Gamma_6}$ contains a Hopf link with three components.
 		\item The fundamental group of $\Omega_{\Gamma_6}$ is not finitely generated.
 	\end{itemize}
 \end{thm}
 
 This situation has a  number of common points with the uniformizable spherical CR structure constructed by Schwartz in \cite[Theorems 1.1 and 1.2]{schwartz_real_2003}. The statement of Schwartz corresponds to a particular representation of the $(4,4,4)$-triangle group, where the images of the standard generators $I_1$, $I_2$ and $I_3$ satisfy $(I_1 I_2 I_1 I_3)^7 = \mathrm{Id}$. If $\Gamma_{(4,4,4,7)}$ is the corresponding subgroup of $\mathrm{PU}(2,1)$ with limit set $\Lambda_{\Gamma_{(4,4,4,7)}}$ and set of discontinuity $\Omega_{\Gamma_{(4,4,4,7)}}$, part of the result can be restated as follows:
 
 \begin{thm}[Schwartz \cite{schwartz_real_2003}] We have that
 	\begin{itemize}
 		\item The group $\Gamma_{(4,4,4,7)}$ is discrete.
 		\item The quotient $\Gamma_{(4,4,4,7)} \backslash \Omega_{\Gamma_{(4,4,4,7)}}$ is a compact orbifold, which is finitely covered by a hyperbolic 3-manifold.
 		\item The limit set $\Lambda_{\Gamma_{(4,4,4,7)}}$ is connected and the closure in $\dhc{2}$ of a countable union of $\RR$-circles.
 		\end{itemize}
 \end{thm}
 
 We believe that the results stated for $\Gamma_6$ also hold for the groups $\Gamma_{3m}$, for $m \geq 2$, defined either as images of representations of $(3,3,3m)$-triangle groups or by the condition $U^{3m} = \mathrm{Id}$ in the Parker-Will parametrization. See \Cref{sect:the_uniformization} for more details on the definition of the groups. In these cases, we still have the string of beads used in \Cref{sect:r_circle_limit_set}, and the proofs should be analogous but with more tedious computations. 
 
 \paragraph{Outline of the article}
 The article is organized as follows. In \Cref{sect:geom_background}, we recall the geometric background of the problem and fix notation for the complex hyperbolic plane $\hc{2}$, its boundary at infinity $\dhc{2}$ and some objects used in the latter constructions. We also prove a result about the shape of uniformizable cusps in spherical CR geometry. 
 In \Cref{sect:the_uniformization}, we briefly describe the manifold that we consider, as well as the subgroup of $\mathrm{SU}(2,1)$ that uniformizes it in $\dhc{2}$. 
 Then, in \Cref{sect:r_circle_limit_set}, we find an $\RR$-circle in the limit set of the group by considering a string of beads, and we prove some consequences of this fact, namely that the limit set is connected, contains a Hopf link and that the fundamental group of its complement is not finitely generated.

\section{Geometric background}\label{sect:geom_background}
 In this section, we will describe briefly a geometric background on the complex hyperbolic plane and its boundary at infinity, the isometry group of this space and the bisectors and spinal spheres, which are the geometric objects that appear in the construction of fundamental domains. For a more complete description, see the book of Goldman \cite{goldman}. We also prove, in Subsection \ref{subsect:cusps_are_horocusps}, a result on the shape of rank-one cusps for spherical CR structures.
 
\subsection{The complex hyperbolic plane and its boundary at infinity}
 Let $V$ be the complex vector space $\CC^3$ endowed with the Hermitian product $\langle \cdot , \cdot \rangle$ given by
 \begin{equation*}
 	\langle z,w \rangle = \conjug{z_1} w_3 + \conjug{z_2} w_2 + \conjug{z_3} w_1
 \end{equation*}
where $z = \left(
\begin{smallmatrix*}
z_1 \\ z_2 \\ z_3
\end{smallmatrix*} 
\right)$
and 
$w = \left(
\begin{smallmatrix*}
w_1 \\ w_2 \\ w_3
\end{smallmatrix*} 
\right)$
 belong to $\CC^3$.
 Let $\Phi$ be the corresponding Hermitian form, of signature $(2,1)$, with matrix
 \begin{equation*}
 	\begin{pmatrix}
 	0 & 0 & 1 \\
 	0 & 1 & 0 \\
 	1 & 0 & 0
 	\end{pmatrix}.
 \end{equation*}
 
 Let $\mathrm{U}(2,1)$ be the unitary group for the Hermitian form $\Phi$, $\mathrm{SU}(2,1)$ its intersection with $\mathrm{SL}_3(\CC)$ and $\mathrm{PU}(2,1)$ its projectivization. In this article, we will use usual brackets to write elements of $\CC^3$ or in a linear group, and square brackets for their projections on $\cp{2}$ and the corresponding projectivized group. For example, if $U \in \mathrm{SU}(2,1)$, then $[U] \in \mathrm{PU}(2,1)$. This last element has exactly three lifts in $\mathrm{SU}(2,1)$, namely $U$, $e^{2i\pi/3} U$ and $e^{-2i\pi/3} U$.
 The complex hyperbolic plane $\hc{2}$ is defined as $\mathbb{P}(\{z \in V \mid \Phi(z)<0 \}) \subset \cp{2}$. The Hermitian form $\Phi$ induces a Riemannian metric on $\hc{2}$ with pinched negative sectional curvature $-1 \leq \kappa \leq -\frac{1}{4}$.
 The boundary at infinity of $\hc{2}$ is the set $\dhc{2} = \mathbb{P}(\{z \in V\setminus \{0\} \mid \Phi(z)=0 \})$.
 In this article, we will use a particular chart to describe $\hc{2}$ and $\dhc{2}$; it is called the \emph{Siegel model} and paramatrizes the spaces as follows:
 \begin{align*}
 	\hc{2} &= \left\{
 	\begin{bmatrix}
 	-\frac{1}{2}(\abs{z}^2 + w)\\
 	z \\
 	1
 	\end{bmatrix}
 	\mid (z,w)\in \CC^2 \text{ and } \Re(w) > 0
 	\right\}, \\
 	\dhc{2} &= \left\{
 	\begin{bmatrix}
 	-\frac{1}{2}(\abs{z}^2 + it)\\
 	z \\
 	1
 	\end{bmatrix}
 	\mid (z,t) \in \CC \times \RR
 	\right\} 
 	\cup 
 	\left\{
 	\begin{bmatrix}
 	1\\
 	0 \\
 	0
 	\end{bmatrix}
 	\right\}.
 \end{align*}
 Thus, the space $\hc{2}$ is homeomorphic to a ball $B^4$, and its boundary at infinity $\dhc{2}$ to a sphere $S^3$, that we identify with $\CC \times \RR \cup \{\infty\}$ in the Siegel model.
 
 The group of holomorphic isometries of $\hc{2}$ is $\mathrm{PU}(2,1)$, and acts transitively on $\hc{2}$ and on $\dhc{2}$. In the same way as for real hyperbolic isometries, an element of $\mathrm{PU}(2,1)$ is \emph{elliptic} if it has a fixed point in $\hc{2}$, \emph{parabolic} if it is not elliptic and has a unique fixed point in $\dhc{2}$ and \emph{loxodromic} otherwise.
 Among the parabolic elements, the \emph{unipotent} ones are precisely those whose lifts in $\mathrm{SU}(2,1)$ wave a triple eigenvalue. We will also use these terms for elements of $\mathrm{SU}(2,1)$, depending on the type of their projection in $\mathrm{PU}(2,1)$.
 
 Given a discrete subgroup $\Gamma < \mathrm{PU}(2,1)$, the \emph{limit set of $\Gamma$} is the set of accumulation points of one (or equivalently any) $\Gamma$-orbit in $\hc{2}$. We will denote this set by $\Lambda_\Gamma$. Thus, $\Lambda_\Gamma$ is a closed, $\Gamma$-invariant subset of $\dhc{2}$. Indeed, it is the smallest closed non-empty $\Gamma$-invariant subset of $\dhc{2}$.
 The complement of $\Lambda_\Gamma$ in $\dhc{2}$ is called the \emph{set of discontinuity} of $\Gamma$, and is denoted $\Omega_\Gamma$. It is the largest open set on which $\Gamma$ acts properly. When the action of $\Gamma$ has no fixed points in $\Omega_\Gamma$, the quotient $\Gamma \backslash \Omega_\Gamma$ is a manifold.

\subsection{Some geometric objects}
 We will use some geometric objects related to the complex hyperbolic plane and its boundary at infinity. First, we focus on the totally geodesic subspaces of $\hc{2}$. Of course, points, geodesics and $\hc{2}$ are totally geodesic. However, there is no totally geodesic subspace of dimension $3$, and there are two types of totally geodesic subspaces of dimension $2$.
 On the one hand, there are the complex geodesics, which are the intersections of complex lines of $\cp{2}$ with $\hc{2}$; they are isometric to $\hc{1}$ and have constant sectional curvature equal to $-1$. On the other hand, there are the real planes, defined as intersections of $\hc{2}$ totally real subspaces of $\cp{2}$. They are copies of $\mathbb{H}^2_\RR$, but the induced distance is rescaled, so they have constant sectional curvature equal to $-\frac{1}{4}$. The group $\mathrm{PU}(2,1)$ acts transitively on each type of subspace.
 
 The boundary at infinity of complex geodesics and real planes are smooth circles in $\dhc{2}$, called \emph{$\CC$-circles} and \emph{$\RR$-circles} respectively. Two of these circles are linked if and only if the corresponding subspaces intersect in $\hc{2}$.
 
 Since there is no totally geodesic hypersurface in $\hc{2}$, we need to consider another kind of geometric objects in order to bound domains. A possible class of objects, that arise naturally when studying Dirichlet domains, are bisectors, which are equidistant surfaces. More precisely, if $[z_1],[z_2] \in \hc{2}$, the \emph{bisector of $[z_1]$ and $[w_2]$} is defined as 
 \[\mathfrak{B}([z_1],[z_2]) = \{ [w] \in \hc{2} \mid d([z_1],[w]) = d([z_2] , [w]) \}.\]
 If $z_1$ and $z_2$ are lifts of $[z_1]$ and $[z_2]$ in $\CC^3$ such that $\Phi(z_1) = \Phi(z_2)$, we define the same object by
 \[\mathfrak{B}(z_1,z_2) = \{[w]\in \hc{2} \mid 
 \abs{\langle z_1 , w \rangle} = 
 \abs{\langle z_2 , w \rangle} \}.\]

 A bisector is homeomorphic to a ball $B^3$. For a more detailed description, see again \cite[Chapters 5, 8 and 9]{goldman}. The boundary at infinity of a bisector is homeomorphic to a sphere $S^2$ in $\dhc{2}$, and is called a \emph{spinal sphere}. For $z,w \in \CC^3$ as before, the corresponding spinal sphere is defined by
 \[\mathfrak{S}(z_1,z_2) = \{[w]\in \dhc{2} \mid 
 \abs{\langle z_1 , w \rangle} = 
 \abs{\langle z_2 , w \rangle} \}.\]
 
 We will consider domains in $\hc{2}$ bounded by bisectors, and their boundaries at infinity, that are bounded by spinal spheres. In general, the intersections and tangencies of bisectors and spinal spheres can be complicated, as Goldman shows in \cite[Chapter 9]{goldman}. However, since we will only consider a Dirichlet domain and bisectors equidistant from a single point, the intersections will be connected.

\subsection{Spherical CR structures and uniformizations}
We will use the language of $(G,X)$-structures in order to work with geometric structures in this article, in the sense given for example by Thurston in \cite{gt3m}.
Given a model space $X$ and a group $G$ acting transitively and analytically on $X$, a $(G,X)$-structure on a manifold $M$ is an atlas of $M$ with values in $X$ and with transition maps given by elements of $G$.
 Equivalently, a $(G,X)$-structure on $M$ can be seen as a pair $(\dev,\rho)$ of a developing map and a holonomy representation, where $\rho: \pi_1(M) \to G$ is a representation and $\dev: \widetilde{M} \to X$ is a $\rho$-equivariant local diffeomorphism, in the sense that for all $\gamma \in \pi_1(M)$ and all $x \in \widetilde{M}$, $\dev(\gamma x) = \rho(\gamma) \dev(x)$.
 In this article we will use both definitions, depending on the points that we want to highlight.
 We will focus here on spherical CR structures, that are defined as follows  

\begin{defn}
	A spherical CR structure on a $3$-manifold $M$ is a $(G,X)$-structure where $X = \dhc{2}$ and $G = \mathrm{PU}(2,1)$.
\end{defn}

 We say that a spherical CR structure on $M$ is \emph{uniformizable} if $\rho(\pi_1(M)) = \Gamma$ is a discrete subgroup of $\mathrm{PU}(2,1)$ with set of discontinuity $\Omega_\Gamma \subset \dhc{2}$, the manifold $M$ is diffeomorphic to $\Gamma \backslash \Omega_\Gamma$ and the spherical CR structure on $M$ is given by this quotient.
 This type of structures is of particular interest, since all the information is contained in the group $\Gamma$.

\subsection{Rank-one spherical CR cusps are horocusps}\label{subsect:cusps_are_horocusps}
Consider a cusped hyperbolic $3$-manifold $M$, and $C$ an open cusp neighborhood bounded by a horosphere, homeomorphic to $T \times \RR^+$, where $T$ is the corresponding peripheral torus.
Choose, once and for all, $\widetilde{C} \subset \widetilde{M}$ to be a copy of the universal cover of $C$ in the universal cover of $M$. By choosing this lift, we obtain an injection of the peripheral group $\pi_1(T) \simeq \pi_1(C) \hookrightarrow \pi_1(M)$.
If $(\dev,\rho)$ is a spherical CR structure on $M$, its \emph{peripheral holonomy} is the restriction of $\rho$ to $\pi_1(T)$.

 If $M$ is one of the cusped hyperbolic manifolds for which a spherical CR uniformization is known, then the image corresponding peripheral holonomy is generated by a single parabolic element.
 In the cases of the Figure Eight knot complement and the Whitehead link complement, the spherical CR cusp has the shape of a \emph{horotube} quotiented by a parabolic element. This follows from the work of Deraux in \cite{deraux_1parameter_2016} for the Figure Eight knot complement and Parker and Will in \cite{parker_complex_2017a} as well as Schwartz in \cite{schwartz} for two different uniformizations of the Whitehead link complement. 
 Recall that, following Schwartz in \cite{schwartz}, a \emph{horotube} and a \emph{horocusp} are defined as follows.
\begin{defn}
	Let $P \in \mathrm{PU}(2,1)$ be a parabolic element with fixed point $[p] \in \dhc{2}$. A \emph{$P$-horotube} is a $P$-invariant open subset $H$ of $\dhc{2} \setminus \{[p]\}$ such that $\langle P \rangle \backslash H$ has a compact complement in $\langle P \rangle \backslash (\dhc{2}) \setminus \{[p]\}$.
	A \emph{$P$-horocusp} is a quotient of the form $\langle P \rangle \backslash H$ where $H$ is a $P$-horotube.
\end{defn}

 In this subsection, we are going to prove that in a spherical CR uniformization, the shape of cusps is always a horocusp, as long as the image of the peripheral holonomy is generated by a single parabolic element. In order to prove this fact, we need a technical lemma, stating that an embedded $P$-invariant cylinder always bounds a horotube.

\begin{lemme}\label{lemma:cylinder_bounds_horotube}
	Let $[P] \in \mathrm{PU}(2,1)$ be a parabolic element with fixed point $[p] \in \dhc{2}$, and let $L \subset \dhc{2}$ be an embedded cylinder invariant by $[P]$. Then $L$ cuts $\dhc{2}\setminus \{[p]\}$ into two connected components $H$ and $K$, where $H$ is a horotube and $\langle P \rangle \backslash (K \cup L)$ is compact.
\end{lemme}
\begin{proof}
	Since $L$ is a closed connected surface in $\dhc{2} \setminus \{[p]\} \simeq \RR^3$, we can apply the Jordan-Brouwer separation theorem (as stated for example in \cite[Theorem 4.16]{montiel_curves_2009}). Hence, $\dhc{2} \setminus (L \cup \{[p]\})$ has at exactly two connected components.
	
	Since $\langle [P] \rangle  \backslash \dhc{2} \setminus \{[p]\} \simeq S^1 \times \RR^2$ and the projection of $C$ is an embedded torus, there is only one unbounded connected component in the quotient, so the two components are a $P$-horotube $H$ and a component $K$ such that $\langle P \rangle \backslash (K \cup L)$ is compact.
\end{proof}

\begin{prop}\label{prop:cusps_are_horocusps}
	Let $(\dev,\rho)$ be a spherical CR uniformization of $M$ such that $\rho(\pi_1(T))$ is a parabolic subgroup generated by a horizontal unipotent element $P$.
	Then, there exists a $P$-horotube $H$ such that the spherical CR structure on $C$ is given by $\langle P \rangle \backslash H$.
\end{prop}
\begin{proof}
	Maybe after choosing a smaller neighborhood $C$, we can suppose that for all $\gamma \in \pi_1(M)$ $\gamma \widetilde{C} \cap \widetilde{C} \neq \emptyset \iff \gamma \in \pi_1(T)$.
	Our goal is to prove that $\dev(\widetilde{C})$ is a horotube $H$, and that the spherical CR structure on $C$ is given by $\langle [P] \rangle \backslash H$.
	
	First, we prove that the structure of $C$ is given by $\langle [P] \rangle \backslash \dev(\widetilde{C})$.
	Since $(\dev,\rho)$ is a uniformization of $M$, we know that the structure on $C$ is given by $\Gamma \backslash \dev(\widetilde{C})$.
	Let $\gamma \in \pi_1(M)$ such that $\rho(\gamma) \dev(\widetilde{C}) \cap \dev(\widetilde{C}) \neq \emptyset$.
	Let $x_0 \in \widetilde{C}$ such that $\rho(\gamma) \dev(x_0) \in \dev(\widetilde{C})$.
	Since $\dev(\widetilde{C})$ is path-connected and $x_0$ and $\gamma x_0$ are lifts of the same point in $M$, there is a path in $\dev(\widetilde{C})$ from $\dev(x_0)$ to $\dev(\gamma x_0)$, that lifts to a path in $\widetilde{C}$ representing an element $\sigma \in \pi_1(T)$.
	Hence, $\rho(\sigma^{-1}\gamma)$ fixes $\dev(x_0)$. Since the actions are continuous, the same is true for points in a neighborhood of $x_0$, so $\sigma^{-1}\gamma \in \ker(\rho)$. Thus, the stabilizer of $\dev(\widetilde{C})$ in $\Gamma$ is $\rho(\pi_1(T))$, so the spherical CR structure on $C$ is given by $\rho(\pi_1(T)) \backslash \dev(\widetilde{C}) = \langle [P] \rangle \backslash \dev(\widetilde{C})$.
	
	It only remains to prove that $\dev(\widetilde{C})$ is a horotube.
	Choose a basis $(l,m)$ of $\pi_1(C)$ such that $m$ generates $\ker(\rho |_{\pi_1(T)})$ and $\rho(l) = [P]$. Since the structure on $M$ is a uniformization, $\dev$ induces an embedding of the cylinder $\langle m \rangle \backslash \partial \widetilde{C}$ into $S^3$. By the equivariance of $\dev$, the cylinder is $[P]$-invariant, so, by \Cref{lemma:cylinder_bounds_horotube}, $\dev(\partial\widetilde{C})$ cuts $\dhc{2} \setminus \{[p]\}$ into two connected components $H$ and $K$, where $H$ is a $[P]$-horotube and $\langle [P] \rangle \backslash (K\cup \dev(\partial\widetilde{C}))$ is compact.
	
	Now, let $\widetilde{N}$ be the full pre-image of $M \setminus \overline{C}$ in $\widetilde{M}$. By the choice of $C$, $\widetilde{N}$ is connected, and disjoint from all the lifts of $C$ in $\widetilde{M}$. Therefore, $\dev(\widetilde{N}) \subset H$ or $\dev(\widetilde{N}) \subset K$. Since $\widetilde{C}$ is connected as well, $\dev(\widetilde{C})$ is contained in the other connected component.
	Note that $\Lambda_\Gamma$ must be contained in the same component as $\dev(\widetilde{N})$.
	Otherwise, since the $\Gamma$-orbit of $[p]$ is dense in $\Lambda_\Gamma$, there would be $[Q] \in \Gamma$ conjugated to $[P]$ with fixed point $[q]$ in the other component. If $[x] \in \dev(\widetilde{N})$, then $[Q]^n [x] \to [q]$, so there would exist $n\in \NN$ such that $[Q]^n[x]$ is in the connected component that does not intersect $\dev(\widetilde{N})$, leading to a contradiction.
	
	Suppose, by contradiction, that $\dev(\widetilde{N}) \subset H$.
	Then, the limit set $\Lambda_\Gamma$ must be contained in $H$, and $K \subset \Omega$.
	Let $\pr: \Omega \to \Gamma \backslash \Omega \simeq M$ be the natural projection.
	Since $\dev(\widetilde{C}) \subset \Omega$, we know that $C \subseteq \pr(K)$, and since $K \cap \pr^{-1}(N) = \emptyset$, we also have $\pr(K) \subseteq C$. Thus, $C \simeq \pr(K) \simeq \Gamma \backslash K$. But $\Gamma \backslash K$ is a quotient of $\langle P \rangle \backslash K$, which is relatively compact by \Cref{lemma:cylinder_bounds_horotube}. The fact that $C$ is not contained in any compact set leads to a contradiction.
	
	Therefore, $\dev(\widetilde{N}) \subset K$, and $\dev(\widetilde{C}) \subset H \subset \Omega$. Since $\dev(\widetilde{N}) \cap H = \emptyset$, $ H = \bigcup_{\gamma \in \Gamma} \dev(\gamma \widetilde{C})$. But $H$ is connected and $\dev(\widetilde{C})$ is either equal or disjoint from its image by an element of $\Gamma$, so $\dev(\widetilde{C}) = H$.

\end{proof}

\begin{cor}
	The spherical CR Dehn surgery theorem of \cite{acosta_spherical_2016} can be applied for all the uniformizazions of cusped hyperbolic manifolds given in \cite{parker_wang_xie}.
\end{cor}

	This last corollary, that follows directly from \Cref{prop:cusps_are_horocusps}, implies that there is an infinite family of hyperbolic cusped manifolds for which an infinite number of Dehn surgeries admit spherical CR structures. Observe, however, that these structures are not necessarily uniformizable.

\section{The uniformization}\label{sect:the_uniformization}
In the rest of this article, we will focus on a particular discrete subgroup of $\mathrm{SU}(2,1)$ that gives a spherical CR uniformization of a one-cusped hyperbolic manifold. We will denote by $\Gamma_6$ or $\Gamma$ the group, and by $M_6$ the manifold that it uniformizes, meaning that $\Gamma$ has a non-empty discontinuity set $\Omega_\Gamma$ and that $M_6 \simeq \Gamma \backslash \Omega_\Gamma$.
The manifold $M_6$ and the group $\Gamma_6$ can be defined in several ways. On the one hand, the manifold $M_6$ is homeomorphic to:
\begin{itemize}
	\item The Dehn surgery of one cusp of the Whitehead link complement of slope $3$. (For the Snappy marking, the peripheral curve that is killed has coordinates $(m,l) = (3,1)$).
	\item The one-punctured torus bundle over $S^1$ with holonomy $
	\left(
	\begin{smallmatrix}
	4 & 3 \\
	1 & 1
	\end{smallmatrix}
	\right)$
	 (named \texttt{b++RRRL} in Snappy).
	\item The manifold \texttt{m023} in the Snappy census of cusped hyperbolic manifolds.
\end{itemize}
On the other hand, the group $\Gamma_6 \subset \mathrm{SU}(2,1)$ is conjugate to:
\begin{itemize}
	\item The index $2$ subgroup of the $(3,3,6)$-triangle group where $I_1 I_2 I_3$ is unipotent, that appears in \cite{parker_wang_xie}.
	\item The Parker-Will representation of $\zz{3}$ of parameter $(\alpha_1,\alpha_2) = (0,\frac{\pi}{3})$, that appears in \cite{parker_complex_2017a}.
	\item A lift in $\mathrm{SU}(2,1)$ of the point in the character variety $\mathcal{X}_{\mathrm{SU}(2,1)}(\zz{3})$ of coordinates $(3,2\cos(\frac{\pi}{3})+1) = (3,2)$, that appears in \cite{acosta_character_2019}
\end{itemize}

We will mainly use the explicit parametrization given by Parker and Will in \cite{parker_complex_2017a}. The group is generated by two order $3$ elements. We will keep the notation of \cite{parker_complex_2017a} and \cite{acosta_spherical_2019}, that we recall briefly.

\begin{notat}
	Let $S,T \in \mathrm{SU}(2,1)$ be the order $3$ generators of $\Gamma_6$, as described in \cite{parker_complex_2017a}. Keeping the same notation, we let $A=ST$ and $B=TS$. We also keep the notation in \cite{acosta_spherical_2019} and let $U=S^{-1}T$ and $V=TS^{-1}$.
	Note that $U$ and $V$ have order $6$ in this group.
The matrices $S$, $T$, $A$ and $B$ are explicitly given by
\begin{align*}
	S &=
	\begin{pmatrix}
	1 & \sqrt{2}\,\conjug{\zeta} & -1 \\
	-\sqrt{2} \zeta & -1 & 0 \\
	-1 & 0 & 0
	\end{pmatrix}
	&
	T &=
	\begin{pmatrix}
	0 & 0 & -1 \\
	0 & -1 & - \sqrt{2}\,\conjug{\zeta}  \\
	-1 &\sqrt{2} \zeta   & 1
	\end{pmatrix}
	\\
	A &=
	\begin{pmatrix}
	1 & -\sqrt{2} & i \sqrt{3} - 1 \\
	0 & 1 & \sqrt{2} \\
	0 & 0 & 1
	\end{pmatrix}
	&
	B &=
	\begin{pmatrix}
	1 & 0 & 0 \\
	\sqrt{2} & 1 & 0 \\
	-i  \sqrt{3} - 1 & -\sqrt{2} & 1
	\end{pmatrix}
\end{align*}
where $\zeta = \exp(\frac{i\pi}{3}) = \frac{1 +i \sqrt{3}}{2}$.
\end{notat}

\begin{rem}
	The coefficients of the Parker-Will representation are in $\QQ[i,\sqrt{2},\sqrt{3}]$
\end{rem}

As in \cite{acosta_spherical_2019}, if $G \in \mathrm{SU}(2,1)$ is a regular elliptic element, we denote by $[p_G]$ its fixed point in $\hc{2}$ and $p_G \in \CC^3$ a lift. In the same way, if $G \in \mathrm{SU}(2,1)$ is parabolic, we denote by $[p_G]$ its fixed point in $\dhc{2}$ and $p_G \in \CC^3$ a lift. There are several possibilities for choosing the lifts in $\CC^3$. However, in our case, if $G_1,G_2 \in \Gamma$ and $G_2$ is elliptic or parabolic, we choose as lift for $[p_{G_1G_2G_1^{-1}}]$ the point $G_1 p_{G_2}$. For example, we can choose the lifts:
\begin{align*}
	p_A &=
	\begin{pmatrix}
	1 \\ 0 \\ 0
	\end{pmatrix}
	&
	p_B &=
	\begin{pmatrix}
	0 \\ 0 \\ 1
	\end{pmatrix}
	&
	p_U &=
	\begin{pmatrix}
	4\\
	-\sqrt{2}(3+i\sqrt{3}) \\
	-4
	\end{pmatrix}
	&
	p_V &=
	\begin{pmatrix}
	4 \\
	\sqrt{2}(1-i\sqrt{3})\\
	-2(1+ i \sqrt{3})
	\end{pmatrix}
\end{align*}

 We consider the bisectors $\mathcal{J}_0^{+} = \mathfrak{B}(p_U,p_V)$,
 $\mathcal{J}_0^{-} = S \mathcal{J}_0^{+}$ and, for $k \in \ZZ/6\ZZ$,
 $\mathcal{J}_k^{\pm} = U^k \mathcal{J}_0^{\pm}$.
 Note that these bisectors are the same as the ones considered by Parker, Wang and Xie in \cite{parker_wang_xie}. The correspondence in the notation for the group is given by $U=I_1I_2$, $S^{-1} = I_1I_3$, $T = I_3I_2$; the corresponding bisectors are $\mathcal{J}_k^{+} = \mathcal{B}_{-2k}$ and $\mathcal{J}_k^{-} = \mathcal{B}_{-2k-1}$.

In \cite{parker_wang_xie}, Parker, Wang and Xie prove that the Dirichlet domain for $\Gamma_6$ centered at $[p_U]$ is bounded by $12$ bisectors, namely $\{\mathcal{J}_k^{\pm} \mid k \in \ZZ/6\ZZ \}$. Using the Poincaré polyhedron theorem, they obtain, as a particular case of \cite[Theorem 1.6]{parker_wang_xie}:
\begin{prop}
	The group $\Gamma_6$ is discrete in $\mathrm{SU}(2,1)$. 
	Furthermore, the domain in $\hc{2}$ bounded by the bisectors $\mathcal{J}_k^{\pm}$ for $k \in \ZZ/6\ZZ$ is the Dirichlet domain of $\Gamma_6$ centered at $[p_U]$. 
	Moreover, $\Gamma_6$ admits the presentation 
	$\langle s,t \mid s^3,t^3,(s^{-1}t)^6 \rangle$.
\end{prop}

Considering the boundary at infinity of the domain and doing some topological considerations, we prove in \cite{acosta_spherical_2019} that it gives a uniformizable spherical CR structure on $M_6$. Thus, we obtain:
\begin{prop}
	The space $\Gamma_6 \backslash \Omega_{\Gamma_6}$ is a manifold homeomorphic to $M_6$.
\end{prop}

We conclude this section by two miscellaneous facts about the subgroup of $\Gamma_6$ generated by $A$ and $B$ and the limit set $\Lambda_{\Gamma_6}$.

\begin{rem}
	As noticed by Parker and Will in \cite[p. 3415]{parker_complex_2017a}, we have $[A,B] = V^3$.
\end{rem}

\begin{prop}\label{prop:AB_finite_index}
	The subgroup $\langle A , B \rangle$ generated by $A$ and $B$ is a normal subgroup of index $3$ of $\Gamma_6$.
\end{prop}
\begin{proof}
	First, we prove that $\langle A , B \rangle$ is a normal subgroup. Since $\Gamma_6$ is generated by $S$ and $T$, which have finite order, we only need to check that $SAS^{-1}$, $SBS^{-1}$, $TAT^{-1}$ and $TBT^{-1}$ belong to $\langle A, B \rangle$.
	But
	$SAS^{-1} = B^{-1} A^{-1}$, 
	$SBS^{-1} = A$, 
	$TAT^{-1} = B$ 
	and $TBT^{-1} = A^{-1} B^{-1}$, so $\langle A, B \rangle$ is a normal subgroup of $\Gamma_6$.
	
	Now, consider the quotient group $\Gamma_6 / \langle A,B \rangle$. It has the presentation $\langle s,t \mid s^3,t^3,(s^{-1}t)^6,st \rangle$, so $t = s^{-1}$, and the presentation simplifies to
	$\langle s \mid s^3 \rangle \simeq \ZZ / 3 \ZZ$.
\end{proof}

Since the limit set of a group is the same as the limit set of any finite index subgroup, we obtain that
the limit set of $\langle A,B \rangle$ is equal to the limit set of $\Gamma_6$.
Some views of this set are pictured in \Cref{fig:views_limit_set}.

\begin{figure}[htb]
	\centering
	\includegraphics[width=0.4\textwidth]{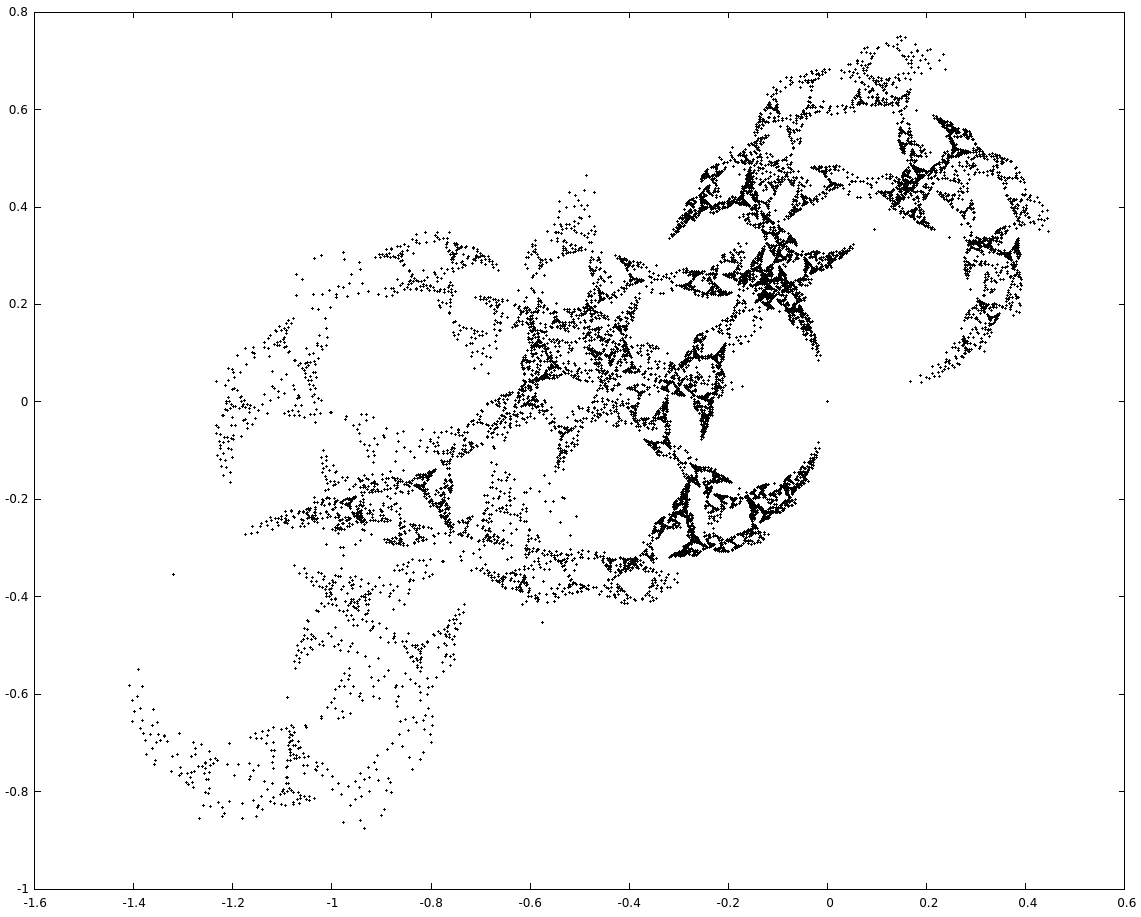}
	\hfill
	\includegraphics[width=0.4\textwidth]{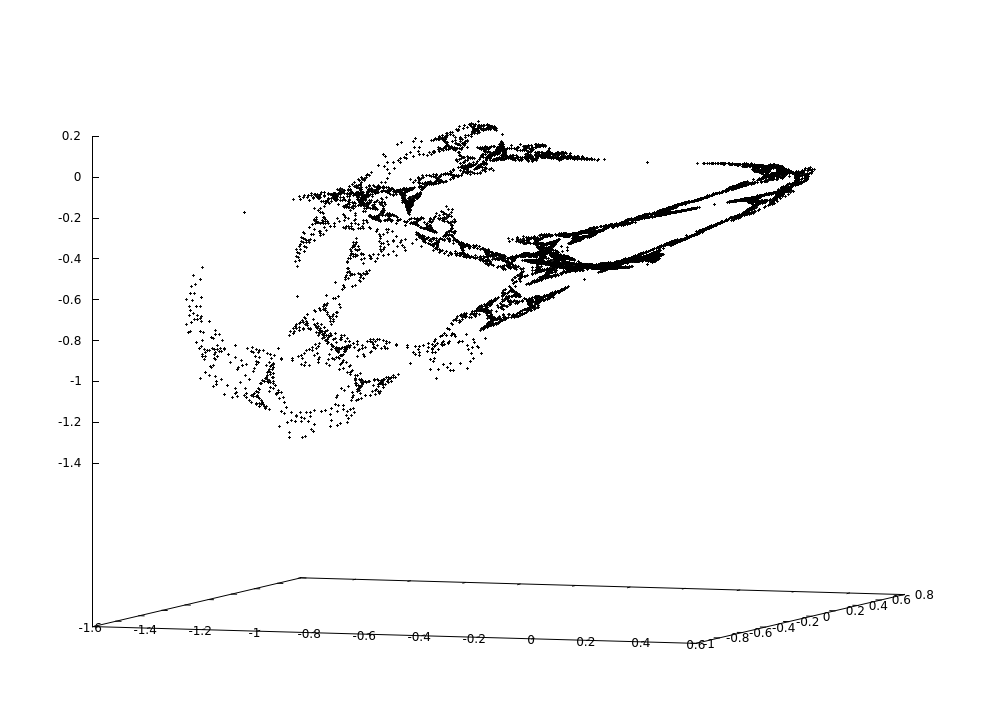} \\
	\includegraphics[width=0.4\textwidth]{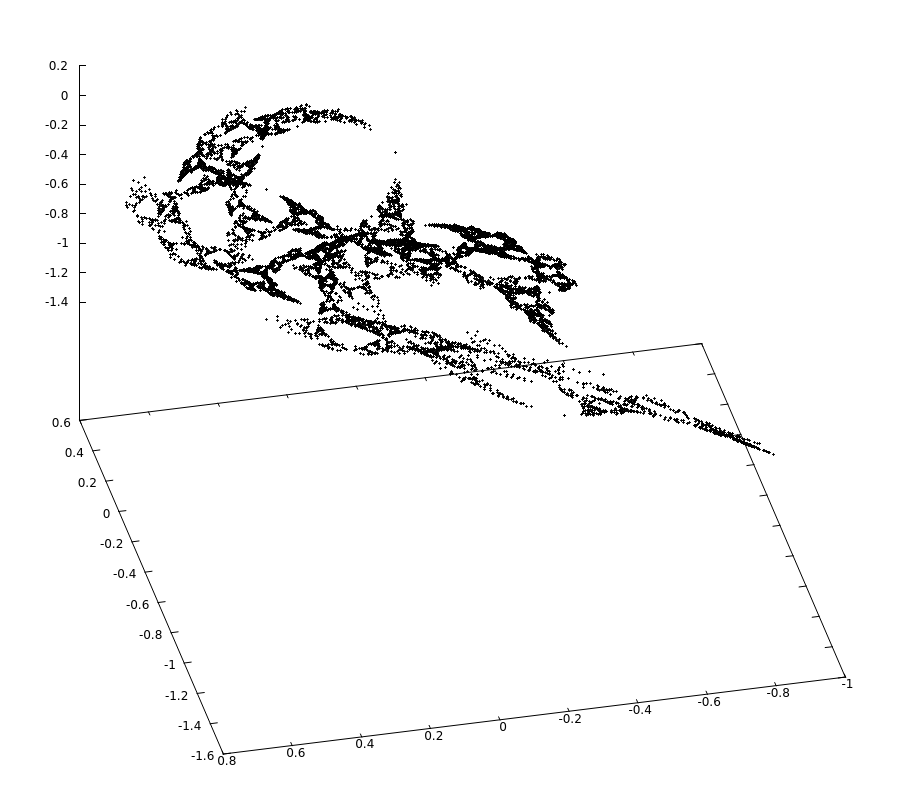}
	\hfill
	\includegraphics[width=0.4\textwidth]{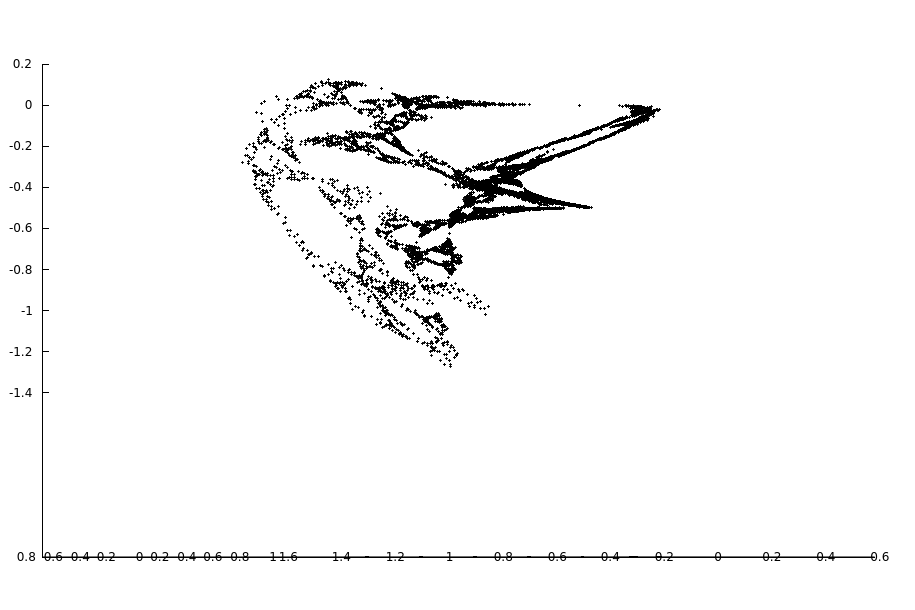}
	\caption{Different views of $\Lambda_{\Gamma_6}$ in the Siegel model.\label{fig:views_limit_set}}
\end{figure}

\section{An $\RR$-circle in the limit set}\label{sect:r_circle_limit_set}
We are going to prove that there is a string of beads related to the action of a subgroup of $\Gamma_6$. With this object, we will be able to prove that $\Lambda_\Gamma$ contains a topological circle, and then prove that the circle is in fact an $\RR$-circle. The fact that $\Lambda_\Gamma$ is the closure of the orbit of this $\RR$-circle by $\Gamma$ follows immediately, which gives us a situation similar to the one in \cite{schwartz_real_2003}. We will then prove three facts about the limit set and the domain of discontinuity of $\Gamma$, namely that the limit set is connected, that it contains a Hopf link and that the fundamental group of the domain of discontinuity is not finitely generated.

\subsection{A string of beads}
  Following the article of Dutenhefner and Gusevskii \cite{dutenhefner_complex_2004}, a
 \emph{string of beads} is
 a finite collection of pairs of spinal spheres $S = \{(S_k,S_k') \mid k \in \{1 , \dots , n \} \}$
  placed along a knot $K$ and satisfying the following condition:
 there is an enumeration $T_1, \dots , T_{2n}$ of the spheres, where the indexes are considered mod $2n$, such that each 
 $T_k$ is tangent to $T_{k \pm 1}$ in an isolated point and
 lies strictly outside all the other spheres.
 
 Let us point out two slight differences with the content of \cite{dutenhefner_complex_2004}, that will not raise any problem. 
 First, observe that the definition above extends the definition of Dutenhefner and Gusevskii from spheres in the Heisenberg group to spheres in $\dhc{2}$, so we should be careful with the meaning of "inside" and "outside" since there is no longer a canonical choice. However, since our spinal spheres are boundaries of bisectors that define a Dirichlet domain, "outside" is to be considered with respect to this domain.
 The other difference is, in our case, that the knot $K$ is unknotted, so there will be no immediate consequences on the limit set. However, this does not change the fact that the limit set of the subgroup that we will consider is a topological circle.
 
 In the following lemma we prove that we have a string of beads made of $4$ bisectors, that have tangency points as in \Cref{fig:string_of_beads}.
 
\begin{figure}[htbp]
	\centering
	\includegraphics[width=10cm]{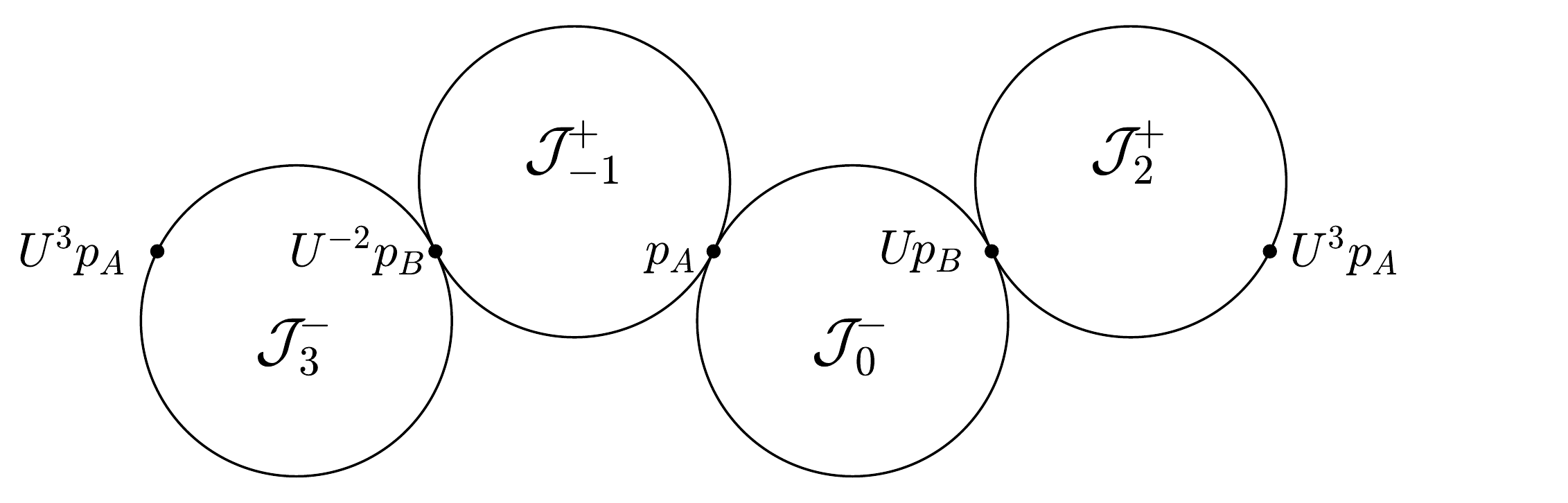}
	\caption{Combinatorics of the string of beads}\label{fig:string_of_beads}
\end{figure}

\begin{lemme}
	The boundaries at infinity of the bisectors $\mathcal{J}_0^{-}$, $\mathcal{J}_{-1}^{+}$, $\mathcal{J}_3^{-}$ and $\mathcal{J}_{2}^{+}$ form a string of beads in $\dhc{2}$, with tangency points $[p_A]$, $U[p_B]$, $U^3[p_A]$ and $U^{-2}[p_B]$, arranged as in \Cref{fig:string_of_beads}.
\end{lemme}
\begin{proof}
	This lemma follows immediately from the considerations done in \cite{acosta_spherical_2019} and \cite{parker_wang_xie} on the incidences and the combinatorics of the bisectors.
	The fact that the points belong to the corresponding bisectors follows immediately from
	\cite[Corollary 8.4]{acosta_spherical_2019}. It only remains to check the tangencies of two consecutive bisectors and the fact that $\mathcal{J}_0^{-} \cap \mathcal{J}_3^{-} = \mathcal{J}_{-1}^{+} \cap \mathcal{J}_{2}^{+} = \emptyset$.
	This is precisely the content of \cite[Theorem 4.3]{parker_wang_xie}.
\end{proof}

\begin{lemme}[Generators for the string of beads]
	The bisector $\mathcal{J}_{-1}^{+}$ is mapped by $A$ to $\mathcal{J}_0^{-}$ and the bisector $\mathcal{J}_{2}^{+}$ is mapped by $U^{3}AU^{-3}$ to $\mathcal{J}_3^{-}$.
	Furthermore, $A(U^{-2}[p_B]) = U[p_B]$ 
	and $U^3AU^{-3}(U[p_B]) = U^{-2}[p_B]$.
\end{lemme}
\begin{proof}
	We only need to prove the first point of each part of the statement; the second one follows immediately by translating by $U^3$.
	We have:
	$A^{-1}\mathcal{J}_0^{-} = A^{-1}\mathfrak{B}(p_U,p_W) 
	=\mathfrak{B}(A^{-1}p_U,A^{-1}p_W)$.
	Now, we compute
	\begin{align*}
		A^{-1}UA &= T^{-1}S^{-1}S^{-1}TST & A^{-1}WA &= T^{-1}S^{-1} STS ST \\
				&= T^{-1}STST & &= S^{-1}T \\
				&= T^{-1}STS^{-1}S^{-1}T & &= U \\
				&= U^{-1}VU
	\end{align*}
	so $A^{-1}p_U = p_{A^{-1}UA} = p_{U^{-1}VU} = U^{-1}p_V$
	and $A^{-1}p_W = p_{A^{-1}WA} = p_U$.
	Hence, \[A^{-1}\mathcal{J}_0^{-}  
	=\mathfrak{B}(A^{-1}p_U,A^{-1}p_W)
	=\mathfrak{B}(U^{-1}p_V,p_U)
	=\mathcal{J}_{-1}^{+}.\]
	For the second point, we have
	\begin{align*}
		AU^{-2}p_B &= STT^{-1}ST^{-1}Sp_B \\
		&= S^{-1}T TS p_B \\
		&= U B p_B \\
		&= Up_B
	\end{align*}
\end{proof}

 With the two previous lemmas we obtain a string of beads consisting of $4$ spinal spheres and whose identifications are given by $A$ and $U^3AU^{-3}$. Then, following the statement of Dutenhefner and Gusevskii in \cite{dutenhefner_complex_2004}, or adapting the argument of Maskit in \cite[VIII F, p. 200]{maskit_kleinian_1988}, we obtain:

\begin{prop}\label{prop:string_of_beads}
	The subgroup of $\Gamma_6$ generated by $A$ and $U^3AU^{-3}$ is a free subgroup. Its limit set is a topological circle.
\end{prop}

In order to simplify a little the computations and be able to express all the elements in terms of $A$ and $B$, we will consider a group which is conjugate to the one defined in \Cref{prop:string_of_beads}. Let $\Gamma' < \Gamma_6$ be the subgroup generated by $B$ and $V^3BV^{-3}$. Since $B = TAT^{-1}$ and $V = TUT^{-1}$, the conclusions of \Cref{prop:string_of_beads} also hold for $\Gamma'$.
Moreover, since $V^3 = [A,B]$ and is of order $2$, we have $V^3BV^{-3} = [A,B]B[A,B] = AB^2A^{-1}B^{-1}$. The limit set of this group is indeed an $\RR$-circle, as proven in the following proposition.

\begin{prop}\label{prop:limit_set_R_circle}
	The limit set of $\Gamma'$ is an $\RR$-circle.
\end{prop}
\begin{proof}
	Let $B_1 = V^3BV^{-3}B = AB^2A^{-1}$, so $\{B,B_1\}$ is a system of generators for $\Gamma'$.
	First, let us characterize the $\RR$-circle $R_0$ that is going to be the limit set. Consider the points 
	\begin{align*}
		[p_B] &=
		\begin{bmatrix}
		0 \\ 0 \\ 1
		\end{bmatrix}
		&
		B_1[p_B] &=
		\begin{bmatrix}
		-16 \\
		4i \sqrt{6} \\
		3 + 6i\sqrt{3}
		\end{bmatrix}
		&
		B_1^{-1}[p_B] &=
		\begin{bmatrix}
		-16 \\
		8\sqrt{2} + 4i \sqrt{6}  \\
		7 + 2i\sqrt{3}
		\end{bmatrix}
	\end{align*}
	
	Observe that the lifts of $[p_B]$, $B_1[p_B]$ and $B_1^{-1}[p_B]$ in $\CC^3$ are linearly independent. We claim that they belong to the same $\RR$-circle.
	We have
		$\langle p_B , B_1p_B \rangle =
		\langle p_B , B_1^{-1}p_B \rangle = 
		-16$
		and
		$\langle B_1p_B , B_1^{-1}p_B \rangle = - 64$,
	so the three points lie in the same $\RR$-circle $R_0$. Furthermore, since the three Hermitian products are real, $R_0$ is the set of points of $\dhc{2}$ that can be written in the form $[x_1p_B + x_2B_1p_B + x_3B_1^{-1}p_B]$ with $x_1,x_2,x_3 \in \RR$.
	
	Now, we prove that $R_0$ is stable by $\Gamma'$. Let $x_1,x_2,x_3 \in \RR$ and $q = x_1p_B + x_2B_1p_B + x_3B_1^{-1}p_B$ be such that $[q] \in R_0$. 
	We only need to check that $B[q]$ and $B_1[q]$ belong to $R_0$.
	A straightforward (but tedious) computation gives:
	\begin{align*}
		B p_B &= p_B \\
		B B_1p_B &= 24p_B + 3B_1p_B -2 B_1^{-1}p_B\\
		B B_1^{-1}p_B &= 8p_B + 2B_1p_B - B_1^{-1}p_B
	\end{align*}
	Hence, $Bq$ is a linear combination of $p_B$, $B_1p_B$ and $B_1^{-1}p_B$ with real coefficients, and therefore $B_1[q]$ belongs to $R_0$.
	In the same way,
	\begin{align*}
	B_1 p_B &= B_1 p_B \\
	B_1^2p_B &= -3p_B - 3B_1p_B + B_1^{-1}p_B\\
	B_1 B_1^{-1}p_B &= p_B
	\end{align*}
	so $B_1q$ is a linear combination of $p_B$, $B_1p_B$ and $B_1^{-1}p_B$ with real coefficients, and therefore $B_1[q]$ belongs to $R_0$. Thus $R_0$ is stable by $\Gamma'$
	
	Consider now the limit set $\Lambda_{\Gamma'}$. Since $B$ is a parabolic element in $\Gamma'$ with fixed point $[p_B]$, we know that $[p_B] \in \Lambda_{\Gamma'}$, so $[p_B] \in R_0 \cap \Lambda_{\Gamma'}$.
	Thus, $R_0 \cap \Lambda_{\Gamma'}$ is a closed nonempty invariant subset for the action of $\Gamma'$ on $\dhc{2}$. Hence $\Lambda_{\Gamma'} \subset R_0$. But, by \Cref{prop:string_of_beads}, we know that $\Lambda_{\Gamma'}$ is a topological circle, so $\Lambda_{\Gamma'} = R_0$.
	 
\end{proof}

 Hence, the limit set of $\Gamma$ is the closure of the orbit of this $\RR$-circle by $\Gamma$, which leads to a situation similar to the one of the one described by Schwartz in \cite[Theorems 1.1 and 1.2]{schwartz_real_2003}, where the limit set is not all $\dhc{2}$ and is the union of a countable set of $\RR$-circles. The following corollary states that there are triples of points of $\Lambda_\Gamma$ lying in the same $\CC$-circle. The proof works as soon as the limit set contains an $\RR$-circle and another point, so th conclusion is also true for the limit set of the Schwartz group of \cite{schwartz_real_2003}.

\begin{cor}
	There exist two distinct points $[q_1],[q_2] \in \dhc{2}$ such that $[p_A]$, $[q_1]$ and $[q_2]$ lie in the same $\CC$-circle.
\end{cor}
	
\begin{proof}
	Consider the Siegel model $\CC \times \RR \cup \{\infty\}$ for $\dhc{2}$, where $[p_A]$ is identified with the point at infinity. By \Cref{prop:limit_set_R_circle}, we know that there is an $\RR$-circle $R_0 \subset \Lambda_\Gamma$ and not passing through $[p_A]$.
	Let $\pi: \CC \times \RR \to \CC$ be the first projection.
	By \cite[Theorem 4.4.7]{goldman} $\pi(R_0)$ is a lemniscate in $\CC$, so there are two distinct points $[q_1], [q_2] \in R_0$ such that $\pi([q_1]) = \pi([q_2]) = z_0$ is the double point of the lemniscate. But $\pi^{-1}(z_0) \cup \{[p_A]\}$ is precisely the $\CC$-circle passing through $[p_A]$ and $(z_0,0)$. Therefore, $[p_A]$, $[q_1]$ and $[q_2]$ lie in the same $\CC$-circle.
\end{proof}

An interesting consequence of this fact is that the group $\Gamma_6$, as well as the group $\Gamma_{(4,4,4,7)}$ considered by Schwartz in \cite{schwartz_real_2003}, cannot be the image of a $(1,1,2)$-hyperconvex representation of a hyperbolic group, as defined by Pozzetti, Sambarino an Wienhard in \cite{pozzetti_conformality_2019}. 

\subsection{The limit set is connected}
 
 Using the $\RR$-circle constructed above, we are going to prove that $\Lambda_\Gamma$ is connected ; the proof will be similar to the one of the connectedness of the limit set described by Schwartz in \cite{schwartz_real_2003}.
 Let $\Gamma'$ be the subgroup of $\Gamma$ generated by $B$ and $V^3 B V^{-3}$, as in the previous subsection. Let $R_0$ be its limit set, which is an $\RR$-circle by \Cref{prop:limit_set_R_circle}. Since $V^3 = [A,B]$ is of order $2$, we have $V^3BV^{-3} = [A,B]B[A,B] = AB^2A^{-1}B^{-1}$, so $\Gamma' = \langle B, AB^2A^{-1} \rangle$ and is in fact a subgroup of $\langle A ,B \rangle$.
 
 \begin{lemme}\label{lemma:R0_cap_CR0_nonempty}
 	Let $C \in \{ A , A^{-1} , B , B^{-1} \}$. Then $R_0 \cap CR_0 \neq \emptyset$.
 \end{lemme}
 \begin{proof}
 	Since $R_0$ is stable by $B$, we have nothing to prove if $C = B$ or if $C = B^{-1}$.
 	For the other cases, observe that $B$ and $AB^2A^{-1}$ are parabolic elements of $\Gamma'$, so their fixed points belong to $R_0$. Since $[p_{AB^2A^{-1}}] = A [p_{B^2}] = A [p_B]$, we have that $\{[p_B] , A[p_B]\} \subset R_0$. 
 	
 	Now, $A R_0$ is the limit set of $A\Gamma' A^{-1} = \langle ABA^{-1}, A^2B^2A^{-2} \rangle$, and therefore contains $A[p_B]$. Hence $AR_0 \cap R_0 \neq \emptyset$.
 	In the same way, $A^{-1} R_0$ is the limit set of $A^{-1}\Gamma' A = \langle A^{-1}BA, B^2 \rangle$, and therefore contains $[p_{B^2}] = [p_B]$. Hence $A^{-1}R_0 \cap R_0 \neq \emptyset$.
 \end{proof}

 \begin{prop}\label{prop:sequence_R_circles}
 	Let $C\in \langle A ,B \rangle$. Then, there exist $m \in \NN$ and $\RR$-circles $R_0, \dots , R_m$ contained in $\Lambda_\Gamma$ such that $[p_B] \in R_0$, $C[p_B] \in R_m$ and for all $0 \leq i \leq m-1$ $R_i \cap R_{i+1} \neq \emptyset$.
 \end{prop}
 \begin{proof}
 	Let $m \in \NN$ such that there exist $C_1, \dots , C_m \in \{A , A^{-1} , B , B^{-1} \}$ with $C = C_1 \cdots C_m$. Let $R_0$ be the limit set of $\Gamma'$ and, for $1 \leq i \leq m$, let $R_i = C_1 \cdots C_i R_0$. Thus, for all $0 \leq i \leq m$, $R_i$ is an $\RR$-circle contained in $\Lambda_\Gamma$, and $C[p_B] \in R_m$. Now, let $i \in \{0, \dots , m-1 \}$. Then, 
 	\begin{align*}
 		R_i \cap R_{i+i} 
 		&= C_1 \cdots C_i R_0 \cap C_1 \cdots C_{i+1} R_0 \\
 		&= C_1 \cdots C_i (R_0 \cap C_{i+1} R_0).
 	\end{align*}
 	By \Cref{lemma:R0_cap_CR0_nonempty}, $R_0 \cap C_{i+1} R_0 \neq \emptyset$, so $R_i \cap R_{i+1} \neq \emptyset$.
 \end{proof}
 
 With the last proposition, we are able establish the connectedness of $\Lambda_\Gamma$.
 
 \begin{prop}
 	The limit set $\Lambda_\Gamma$ is connected and the closure in $\dhc{2}$ of a countable union of $\RR$-circles.
 \end{prop}
 \begin{proof}
 	By \Cref{prop:AB_finite_index}, we know that $\langle A,B \rangle$ has finite index in $\Gamma$. Since $[p_B] \in \Lambda_\Gamma$, the orbit of $[p_B]$ by $\langle A,B \rangle$ is therefore dense in $\Lambda_\Gamma$. Hence, since $[p_B] \in R_0 \subset \Lambda_\Gamma$, the limit set $\Lambda_\Gamma$ is the closure in $\dhc{2}$ of the $\Gamma$-orbit of $R_0$, so it is the closure of a countable union of $\RR$-circles.
 	Now, by \Cref{prop:sequence_R_circles}, if $C \in \langle A,B \rangle$, there is a path in $\Lambda_\Gamma$ from $[p_B]$ to $C[p_B]$. Hence, the set $\Gamma R_0$ is path-connected. Since $\Lambda_\Gamma$ is the closure of $\Gamma R_0$, it is connected.
 \end{proof}

\subsection{A Hopf link in the limit set}
 It only remains to establish two topological facts about the limit set and the domain of discontinuity of $\Gamma$. We begin by proving that the limit set contains linked $\RR$-circles, and therefore a Hopf link with three components.
  
\begin{lemme}\label{lemma:R-circle_linked_with_axes_of_V}
	The $\RR$-circle $R_0$ is linked with the two invariant $\CC$-circles for $V$.
\end{lemme}
\begin{proof}
	Consider the point $[p_V] = \left[\begin{smallmatrix}
	4 \\
	\sqrt{2} -i  \sqrt{6}  \\
	-2 -2 i \sqrt{3} 
	\end{smallmatrix}\right]$. 
	We have that $p_V$ is an eigenvector with eigenvalue $1$ for $V$, and $\langle p_V,p_V \rangle = -8$, so $[p_V] \in \hc{2}$ is the fixed point of $V$ in $\hc{2}$.
	
	We also have that $p_V = -\frac{7}{4}p_B - \frac{3}{8}B_1p_B + \frac{1}{8}B_1^{-1}p_B$, so it belongs to the $\RR$-plane spanned by $p_B$, $B_1p_B$ and $B_1^{-1}p_B$. Thus, both axes of $V$ intersect the $\RR$-plane at $[p_V]$, so, at the boundary at infinity, $R_0$ is linked with the boundaries of these complex axes, which are precisely the two invariant $\CC$-circles for $V$.
\end{proof}

\begin{prop}
	The limit set $\Lambda_{\Gamma}$ contains a Hopf link with three components.
\end{prop}
\begin{proof}
	Let $R_0$ be the $\RR$-circle which is the limit set of the subgroup $\Gamma' = \langle B, V^3BV^{-3} \rangle$, as in \Cref{prop:limit_set_R_circle}.
	Let $R_1$ and $R_2$ be the $\RR$-circles $V R_0$ and $V^2R_0$ respectively, which are the limit sets of the groups $V\Gamma'V^{-1}$ and $V^2\Gamma'V^{-2}$.
	By \Cref{lemma:R-circle_linked_with_axes_of_V}, we know that $R_0$ is linked with the axes of $V$. Hence, the circles $R_0$, $R_1$ and $R_2$ form a Hopf link with three components.
	\end{proof}

\subsection{The fundamental group of $\Omega_\Gamma$ is not finitely generated.}	
	Finally, we will use the $\RR$-circles of the limit set and the fact that $M_6$ is a cusped manifold to prove that $\pi_1(\Omega_\Gamma)$ is not finitely generated.
	Since $M_6$ is a cusped hyperbolic manifold uniformized by $\Gamma_6$, 
	\Cref{lemma:cylinder_bounds_horotube} gives immediately:
\begin{lemme}
	The limit set $\Lambda_\Gamma$ is contained in the complement of a $B$-horotube $H$ based at $[p_B]$.
\end{lemme}

Let $\gamma$ be a loop going once around the horotube $H$ of the previous lemma. We prove the two following lemmas in order to have all the tools for showing that $\pi_1(\Omega_\Gamma)$ is not finitely generated.

\begin{lemme}
	The curve $\gamma$ is linked once around $R_0$ and it is nontrivial in $\pi_1(\Omega_\Gamma)$.
\end{lemme}
\begin{proof}
	Since $R_0 \subset \Lambda_\Gamma$, it is in the core of complement of $H$, so $\gamma$ is linked once around $R_0$. Therefore, it is not homotopically trivial in $S^3 \setminus R_0$, so neither in $\Omega_\Gamma$. 
\end{proof}

\begin{lemme}\label{lemma:homotopic_loop_intersects_D0}
	Let $D_0$ be an open disk in $\dhc{2}$ whose boundary is $R_0$ and that intersects $R_1$ at exactly one point $[p_0] \neq V[p_B]$. Then, any loop that is freely homotopic to $\gamma$ in $\Omega_\Gamma$ intersects $D_0$.
\end{lemme}
\begin{proof}
	Let $\gamma'$ be a loop that is freely homotopic to $\gamma$ in $\Omega_\Gamma$, and suppose that $\gamma' \cap D_0 = \emptyset$.
	Since $\overline{D_0}$ is a closed disk in $\dhc{2} \simeq S^3$, its complement is simply connected. Hence, $\gamma'$ is homotopically trivial in $\dhc{2} \setminus \overline{D_0}$, and therefore in $\dhc{2} \setminus R_0$. But $\gamma$ is linked with $R_0$ and freely homotopic to $\gamma'$ in $\Omega_\Gamma \subset \dhc{2} \setminus R_0$, which is a contradiction.
\end{proof}
	
\begin{prop}
	The fundamental group of $\Omega_\Gamma$ is not finitely generated.
\end{prop}
\begin{proof}
	We prove the proposition by contradiction. Suppose that $\pi_1(\Omega_\Gamma)$ is finitely generated.
	Then, by Scott's compact core theorem (see \cite{scott_compact_1973}), there is a compact 3-manifold $N \subset \Omega_\Gamma$ such that the induced map $\pi_1(N) \to \pi_1(\Omega_\Gamma)$ is an isomorphism.
	
	Let $C= VBV^{-1}$ be a horizontal unipotent element in $\Gamma$, with fixed point $V[p_B]$. We know that $V[p_B] \notin \overline{D_0} = D_0 \cup R_0$ by construction. Since $\overline{D_0}$ is compact, there is an open neighborhood $\mathcal{U}$ of $V[p_B]$ in $\dhc{2}$ such that $\mathcal{U} \cap \overline{D_0} = \emptyset$.
	Since $N$ is compact and does not contain $V[p_B]$ (because $V[p_B] \notin \Omega_\Gamma$), there is $n \in \NN$ such that $C^n N \subset \mathcal{U}$.
	But $C^n$ is an automorphism of $\Omega_\Gamma$, so $C^nN$ satisfies the conclusion of the Scott theorem. Hence, the natural map $\pi_1(C^nN) \to \pi_1(\Omega
	_\Gamma)$ is still an isomorphism. Thus, there is a loop $\gamma_1 \subset C^n N$ that is freely homotopic to $\gamma$ in $\Omega_\Gamma$. But $C^n N \subset \mathcal{U}$, that has empty intersection with $D_0$.
	Hence, $\gamma_1 \cap D_0 = \emptyset$, which contradicts \Cref{lemma:homotopic_loop_intersects_D0}. 
\end{proof}

\bibliographystyle{alpha}
\bibliography{limit_sets_cr}

\end{document}